\documentclass[aop]{imsart}

\usepackage[margin=1in]{geometry}

\usepackage{amsthm,amsmath}

\usepackage[usenames,dvipsnames]{color}
\usepackage{xr}
\usepackage[extension=pdf]{hyperref}
\usepackage{mathrsfs}

\arxiv{arXiv:1206.1821}

\startlocaldefs

\endlocaldefs

\newtheorem{theorem}{Theorem}

\newtheorem{lemma}{Lemma}
\newtheorem{remark}{Remark}

\newcommand{\ze}{\mathbf{Z}}

\newcommand{\re}{\mathbf{R}}

\newcommand{\om}{\omega}

\newcommand{\e}{\varepsilon}

\newcommand{\p}{\mathbf{P}}

\newcommand{\esp}{\mathbf{E}}

\newcommand{\wh}{\mathscr{W}}

\newcommand{\she}{\mathcal{Z}}

\newcommand{\ba}{\omega}
\newcommand{\bg}{\omega}

\begin{document}

\begin{frontmatter}

\title{On the (strict) positivity of solutions of the stochastic heat equation}
\runtitle{Positivity of the SHE}


\author{\fnms{Gregorio R.} \snm{Moreno Flores}\ead[label=e1]{grmoreno@mat.puc.cl}\thanksref{t1}}
\thankstext{t1}{The author acknowledges the support of Fondecyt grant number 1130280 and Iniciativia Cientifica Milenio grant number NC130062 }
\address{\printead{e1} \\ Departamento de Matem\'aticas \\ Vicu\~na Mackenna 4860 \\ Santiago, Chile}
\affiliation{Pontificia Universidad Cat\'olica de Chile}

\runauthor{Gregorio R. Moreno Flores}

\begin{abstract}
We give a new proof of the fact that the solutions of the stochastic heat equation, started with non-negative initial conditions, are strictly positive at positive times. The proof uses concentration of measure arguments for discrete directed polymers in Gaussian environments, originated in M. Talagrand's work on spin glasses and brought to directed polymers by Ph. Carmona and Y. Hu. We also get slightly improved bounds on the lower tail of the solutions of the stochastic heat equation started with a delta initial condition.
\end{abstract}




\end{frontmatter}


A very well known theorem proved by Mueller insures the strict positivity of the solution of the Stochastic Heat Equation (SHE) with non-negative initial data \cite{Mueller}. 

Mueller's theorem has gained new attention due to the links between the SHE and the Continuum
Directed Polymer (CDP) \cite{AKQ2}, and, more generaly, with the KPZ equation (see the review
\cite{Ivan}). In particular, it implies the positivity of the partition function of the CDP. This
random measure on directed paths from $(0,0)$ to $(T,x)$ is defined by
\begin{eqnarray*}
 \mu_{x,T}(X_{t_1}\in dx_1,\cdots,X_{t_k}\in dx_k)= \frac{1}{\she(0,0;T,x)} 
\prod^{k-1}_{j=0}\she(t_j,x_j;t_{j+1},x_{j+1}) \, \she(t_k,x_k;T,x) dx_1 \cdots dx_k,
\end{eqnarray*}
where $\she(s,u;t,v)$ is obtained as the solution of
\begin{eqnarray*}
 \partial_t \she(s,u;\cdot,\, \cdot)&=& \tfrac12 \Delta  \she(s,u;\cdot,\, \cdot) + \she(s,u;\cdot,\, \cdot) \wh, \\
 \she(s,u;s,\cdot) &=& \delta_u(\cdot).
\end{eqnarray*}
The SHE arises as the limit of the renormalized partition function of discrete directed polymers \cite{AKQ} and the CDP as the weak limit of the discrete directed polymer path measure (see \cite{CSY} for a general review on directed polymers).

A proof of the positivity of the solutions of the SHE contained inside the theory of directed polymers is hence desirable and this is the approach we will follow in this note. Our proof, together with providing a more straightforward argument, also improves existing bounds on the tails of the solution of the SHE. Our methods are strongly inspired by Talagrand's use of Gaussian concentration in spin glasses, and Carmona-Hu \cite{CH} where these ideas are applied to directed polymers in a Gaussian environment.

\section{Results}
In the following, unless stated otherwise, $\she(t,x)$ is the continuous modification of the solution of the stochastic heat equation
 \begin{eqnarray}\label{she-delta1}
  \partial_t \she \, &=& \tfrac12 \Delta  \she + \she \wh,\\ 
  \label{she-delta2}
  \she(0,x) &=& \delta_0(x),
 \end{eqnarray}
where $\wh$ is a space-time white noise.
\begin{theorem}\label{main-th}
$a)$ There exists a locally bounded function $c(t,x)>0$, locally bounded away from $0$, such that
 \begin{eqnarray}\label{she-tail-x}
  \p\left[ \she(t,x) < c(t,x)\,e^{-u/c(t,x)}\right]\leq e^{-u^2/2},
 \end{eqnarray}
 hence, for all $p>0$, there is a locally bounded function $\kappa_p(t,x)>0$, locally bounded away
from $0$, such that,
\begin{eqnarray}\label{moments}
 \esp \she(t,x)^{-p}\leq \kappa_p(t,x)\, \exp\{ \, p^2/\kappa_p(t,x) \}, \quad \forall \, t>0,\, x\in \re.
\end{eqnarray}
\noindent $b)$ We have 
\begin{eqnarray}\label{all-x-positive}
\p[\, \she(t,x)=0,\, \text{for some}\,\, t>0,\, x\in \re \,]=0.
\end{eqnarray}
\end{theorem}
\begin{remark}\rm A few remarks are in order:
\begin{enumerate}
\item We note that, in \cite{MuNu}, an estimate similar to (\ref{she-tail-x}) is proved (in a
slightly different context), but the right hand side is $\exp\{-u^{3/2-\e}\}$. Based on the links
between KPZ and random matrices (see for instance \cite{Ivan}), it is reasonable to expect that
the optimal bound in our setting is $\exp\{-u^3\}$. Our bound $\exp \{-u^2\}$ comes from Gaussian
concentration arguments and is unlikely to be improved with our methods.
\item Theorem \ref{main-th} b) for general positive initial datas can be obtained by integrating the solution of (\ref{she-delta1})-(\ref{she-delta2}) against the initial conditions, together with comparison arguments with respect to the initial conditions (see \cite{MuNu}). Theorem \ref{AKQ} below, which provides the convergence of partition functions of directed polymers to the SHE, can in fact be extended to provide convergence for general initial datas by introducing boundary values for the polymer (see \cite{MQR}). Then, the aformentionned comparison arguments can be obtained very easyly, noting that, at the discrete level, they hold path-by-path and are preserved by taking weak limit.
\item With a bit of work, Theorem \ref{AKQ} can also be extended to cover the case of the SHE
 \begin{eqnarray*}
  \partial_t \she \, &=& \tfrac12 \Delta  \she + b \she + \sigma \wh \she,
 \end{eqnarray*}
 for a bounded drift $b=b(t,x)$ and some nice $\sigma=\sigma(t,x)$.  The drift can be handled using standard comparison arguments (see \cite{MuNu}, proof of Theorem 2, where this argument is presented) and the arguments of our proof will also follow with minor modifications. Again, the comparison argument can be obtained very easily from directed polymers. 
\end{enumerate}
\end{remark}
\noindent The proof of Theorem \ref{main-th} using concentration of measure is given in Section \ref{proof}. Section \ref{pre} provides useful preliminaries, while the technical estimates are deferred to the appendix.
\section{Some preliminaries}\label{pre}
\subsection{Directed polymers and the AKQ theory} \label{sec-AKQ}
Let $P$ be the law of the simple symmetric random walk $S_t$ on $\ze$, let $\{\om(i,x):\, i,x\}$ be a collection of real numbers (the environment) and let
\begin{eqnarray}
 Z_{N}(\omega,\beta,x) = E\left[ e^{\beta \sum^N_{i=1}\om(i,S_i)}|\, S_N=x\right],
\end{eqnarray} 
be the partition function of the directed polymers in environment $\om$ at inverse temperature
$\beta>0$, where $E$ denotes expectation with respect to $P$. In the following, we will often denote
$Z_N(\om,\beta)=Z_N(\om,\beta,x)$, or even $Z_N(\beta)=Z_N(\om,\beta,x)$, when no confusion is
possible. In this paper, the $\om$'s are chosen to be independent standard normal random variables.
We denote the law of the environment by $\p$ and expectation with respect to $\p$ by $\esp$. In this
case, $\esp Z_{N}(\omega, \beta,x)= \exp \{\tfrac{N}{2} \beta^2\}$. Define
\begin{eqnarray}
 \she_N(t,x) := e^{-\tfrac12 t \sqrt{N}} \, Z_{tN}(\omega,N^{-1/4},x\sqrt{N}) =
\frac{Z_{tN}(\omega,N^{-1/4},x\sqrt{N})}{\esp Z_{tN}(\omega,N^{-1/4},x\sqrt{N})}. 
\end{eqnarray}
The following theorem by Alberts-Khanin-Quastel (AKQ) shows the scaling limit of the partition
function to the solutions of the stochastic heat equation:
\begin{theorem}\label{AKQ}
\cite{AKQ} For each $t>$ and $x\in \re$, we have the convergence in law,
 \begin{eqnarray}
  \she_N(t,x) \Rightarrow \sqrt{4\pi} e^{\frac{\, x^2}{4t}} \she(2t,x),
 \end{eqnarray}
 where $\she$ is the solution of (\ref{she-delta1})-(\ref{she-delta2}). Furthermore, the convergence holds at the process level in $t$ and $x$.
\end{theorem}
\subsection{Gaussian concentration} We borrow the following from \cite{T} (Lemma 2.2.11). Let
$d(\cdot,\cdot)$ denote the euclidean distance.
\begin{theorem}[Talagrand] 
 Let $\bg$ be an $\re^m$-valued Gaussian vector with covariance matrix $I$, the identity matrix in
$\re^m$. Then, for any measurable set $A\subset \re^m$, if $\p[\, \bg \in A\,]\geq c>0$, then, for
any $u>0$,
\begin{eqnarray}
 \p\left[ d(\bg,A) > u + \sqrt{2\log(1/c)}\,\right]\leq e^{-\tfrac{u^2}{2}}.
\end{eqnarray}
\end{theorem}
\noindent The distance appears naturally when we compare the partition function over different
environments. First, define the polymer measure in a fixed environment $\om$ by
\begin{eqnarray}
  \langle F(S)\rangle_{N,\om,x}&=& \frac{1}{Z_N(\om,\beta,x)} E\left[ F(S)\,e^{\beta \sum^N_{i=1} \om(i,S_i)}|\, S_N=x\sqrt{N}\right].
\end{eqnarray}
We will denote $\langle F(S)\rangle_{N,\om}=\langle
F(S)\rangle_{N,\om,x}$ when no confusion is possible.
Denote the expected value over two independent copies of the polymer in the same environment by
$\langle \cdot \rangle^{(2)}_{N,\ba, x}$ and, for two paths $S^{(1)}$ and $S^{(2)}$, let
$L_N(S^{(1)},S^{(2)})=\sum^N_{t=1} {\bf 1}_{S^{(1)}_t=S^{(2)}_t}$ be the overlap. 
Let $d_N(\ba,\ba')$ denote the euclidean distance between two environments $\ba$ and $\ba'$ when
they are considered as vectors with coordinates in the cone $\{(t,x):\, 0\leq t \leq N, |x|\leq
t\}$.
The proof of next Lemma can be found in Carmona-Hu \cite{CH}, page 443, as part of the proof of their Theorem
$1.5$.
\begin{lemma}\label{two-environments}
Let $\ba$ and $\ba'$ be two environments. Then,
\begin{eqnarray}
\log Z_N(\ba',\beta,x)\geq \log Z_N(\ba,\beta,x) - \beta \, d_N(\ba,\ba') \sqrt{\langle L_N(S^{(1)},S^{(2)})\rangle^{(2)}_{N,\ba,x}}.
\end{eqnarray}
\end{lemma}
\section{Proof of Theorem \ref{main-th}}\label{proof}
Fix $x$ and let $E^{(2)}_{x,N}$ denote the expected value with respect to two independent walks of length $N$ conditioned to end at $x\sqrt{N}$. 
Define the event
\begin{eqnarray}
 A=\left\{ \ba:\, Z_N(\ba,\beta,x)\geq \frac12 \esp Z_N(\beta,x),\,  \langle L_N(S^{(1)},S^{(2)})\rangle^{(2)}_{N,\ba,x} \leq C\sqrt{N}\right\},
\end{eqnarray}
Versions of the following Lemma for fixed $\beta$ can be found in \cite{T}, Lemma 2.2.9, for spin glasses,
and in \cite{CH}, proof of Theorem $1.5$, for directed polymers.
\begin{lemma}\label{positive-A}
Take $\beta=N^{-1/4}$. For $C>0$ large enough, there exists $\delta>0$ such that $\p[A]\geq \delta,\, \forall \, N\geq 1$. Furthermore, $\delta$ can be taken uniformly bounded away from $0$ for $x$ in a compact set.
\end{lemma}
\begin{proof} 
The key to prove this fact is the estimate (\ref{overlap2}) proved in Section \ref{estimates-overlap}. Let $H_N(S^{(1)},S^{(2)})=\sum^N_{t=1} \om(t,S^{(1)}_t)+\om(t,S^{(2)}_t)$.
\begin{eqnarray}\nonumber
 \p[A] &=& \p\left\{ Z_N(\beta,x)\geq \frac12 \esp Z_N(\beta,x),\,  E^{(2)}_{x,N}\left[ L_N(S^{(1)},S^{(2)}) \exp\{ \beta
H_N(S^{(1)},S^{(2)})\}\right] \leq C\sqrt{N} Z_N(\beta,x)^2 \right\}\\
 &\geq& \p\left\{ Z_N(\beta,x)\geq \frac12 \esp Z_N(\beta,x),\,  E^{(2)}_{x,N}\left[ L_N(S^{(1)},S^{(2)}) \exp\{ \beta
H_N(S^{(1)},S^{(2)})\}\right] \leq \frac{C}{4}\sqrt{N} \left(\esp Z_N(\beta,x)\right)^2 \right\}\\
 &\geq& \p\left\{ Z_N(\beta)\geq \frac12 \esp Z_N(\beta)\right\} \\
 && \quad -\,\p\left\{E^{(2)}_{x,N}\left[ L_N(S^{(1)},S^{(2)}) \exp\{ \beta H_N(S^{(1)},S^{(2)})\}\right] > \frac{C}{4}\sqrt{N} \left(\esp Z_N(\beta)\right)^2 \right\}
\end{eqnarray}
We treat the first summand: by Paley-Zygmund's inequality (see for example \cite{T}, Proposition 2.2.3),
\begin{eqnarray}
\p\left\{ Z_N(\beta,x)\geq \frac12 \esp Z_N(\beta,x)\right\} \geq \frac14 \frac{\left( \esp Z_N(\beta,x)\right)^2}{\esp Z_N(\beta,x)^2} = \frac{1}{4}\frac{1}{\esp \she^2_N(1,x)},
\end{eqnarray}
if we take $\beta=N^{-1/4}$. Now, by an application of Fubini's theorem together with $\esp e^{\beta \om}=e^{\beta^2/2}$ (remember $\om$ is a standard normal random variable), we have $\esp \she^2_N(1,x)=E^{(2)}_{x,N}[\exp N^{-1/2}L_N(S^{(1)},S^{(2)})]$. The estimate (\ref{overlap1}) then provides a constant $0<L<+\infty$ such that
 \begin{eqnarray}
   \esp \she^2_N(1,x) \leq L,\quad \, \forall \, N\geq 1.
 \end{eqnarray}
 This gives
 \begin{eqnarray}
\p\left\{ Z_N(\beta,x)\geq \frac12 \esp Z_N(\beta,x)\right\} \geq  \frac{1}{4L},\quad \forall \, N\geq 1 \quad \text{when} \quad \beta=N^{-1/4}.
\end{eqnarray}
For the second summand above, using Chebyshev followed by Fubini
\begin{eqnarray}
&&\p\left\{ E^{(2)}_{x,N}\left[ L_N(S^{(1)},S^{(2)}) \exp\{ N^{-1/4}H_N(S^{(1)},S^{(2)})\}\right] > \frac{C}{4}\sqrt{N} \left(\esp Z_N(\beta,x)\right)^2\right\}\\
&& \phantom{blabla}\leq \frac{4}{C\sqrt{N}\left(\esp Z_N(\beta,x)\right)^2} \esp E^{(2)}_{x,N}\left[ L_N(S^{(1)},S^{(2)}) \exp\{ N^{-1/4}H_N(S^{(1)},S^{(2)})\}\right]\\
&& \phantom{blabla}= \frac{4}{C\sqrt{N}} E^{(2)}_{x,N}\left[ L_N(S^{(1)},S^{(2)}) \exp\{ N^{-1/2}L_N(S^{(1)},S^{(2)})\}\right]\\
&& \phantom{blabla}\leq \frac{4 K}{C},
\end{eqnarray}
for some $K>0$, thanks to (\ref{overlap2}), where we also used
\begin{eqnarray}\nonumber
\esp E^{(2)}_{x,N}\left[ L_N(S^{(1)},S^{(2)}) \exp\{ N^{-1/4}H_N(S^{(1)},S^{(2)})\}\right]
= \left(\esp Z_N(\beta,x)\right)^2 \, E^{(2)}_{x,N}\left[ L_N(S^{(1)},S^{(2)}) \exp\{ N^{-1/2}L_N(S^{(1)},S^{(2)})\}\right].
\end{eqnarray}
Overall, we have $\p[A] \geq \frac{1}{4L} - \frac{4 K}{C}=:\delta$, which is positive provided we choose $C$ large enough. Finally, note that the constants $L$ and $K$ can be chosen uniformly bounded for $x$ in a compact set.
\end{proof} 

\begin{proof}[Proof of Theorem \ref{main-th}- $a)$] Recall the distance $d_N(\cdot,\cdot)$ from Lemma \ref{two-environments}. By Lemma \ref{positive-A} and Talagrand's theorem,
\begin{eqnarray}
 \p\left[\bg:\, d_N(\bg, A)>u+C'\right] \leq e^{-u^2/2}, 
\end{eqnarray}
for all $u>0$ and some explicit constant $0<C'<+\infty$ depending on $C,\, K$ and $L$. 
In particular, for any $\ba' \in A$, if $\bg$ is any environment, by Lemma \ref{two-environments},
\begin{eqnarray}
 \log Z_N(\bg,\beta,x) &\geq& \log \esp Z_N(\beta,x)-\log 2 - \beta \, d_N(\ba,\bg') \sqrt{\langle L_N(S^{(1)},S^{(2)})\rangle^{(2)}_{N,\ba'}}\\
 &\geq& \log \esp Z_N(\beta,x)-\log 2 - \beta N^{1/4}\sqrt{C} d_N(\ba,\bg'),\\
  &\geq& \log \esp Z_N(\beta,x)-\log 2 - C'' d_N(\ba,\bg'),
\end{eqnarray}
for some $0<C''<+\infty$, if $\beta=N^{-1/4}$. As a consequence, if $\log Z_N(\bg,\beta,x) \leq \log \esp Z_N(\beta,x)-c_2 u-c_1$, then
$$\log \esp Z_N(\beta,x)-\log 2 - C'' d_N(\ba,\bg')\leq \log \esp Z_N(\beta,x)-c_2 u-c_1.$$
Taking $c_2=C''$ and $c_1=\log 2 + C'C''$, this in turns implies that $d(\bg,\bg')\geq u+C'$ for all $\bg'\in A$ and
\begin{eqnarray}
&& \p\left[ \log Z_N(\bg,\beta,x) \leq \log \esp Z_N(\beta,x)-c_2 u-c_1\right] \,\,  \leq \,\, \p\left[ d_N(\bg,A) \geq u+C'\right] \,\, \leq \,\, e^{-u^2/2}.
\end{eqnarray}
This proves the following intermediate result: for all $u>0$, $N\geq 1$, (remember $\she_N(1,x) = Z_N(x)/\esp Z_N(x)$)
 \begin{eqnarray}
  \p\left[ \she_N(1,x) < C_2e^{-c_2u}\right]\leq e^{-u^2/2}.
 \end{eqnarray}
with $C_2=e^{-c_1}$. Using that $\she_N(1,x) \to \sqrt{4\pi } e^{x^2/4}\she(2,x)$ in law, we get
 \begin{eqnarray}\label{she-tail}
  \p\left[ \she(2,x) < C_2 (4 \pi)^{-1/2} e^{-x^2/4}e^{-c_2u}\right]\leq e^{-u^2/2}.
 \end{eqnarray}
 for all $u>0$. This proves Theorem \ref{main-th}-$a)$ when $t=2$. If we take the length of the polymer to be $tN$, the proof is unchanged, and the estimates of Section \ref{estimates-overlap} imply that the constants $C'$ and $C''$ above are uniformly bounded for $(t,x)$ in a compact set.
\end{proof}
 \begin{proof}[Proof of Theorem \ref{main-th}- $b)$]
 We will use the following standard estimate: for any $p>1$ and any compact set $K$, there exists a constant $C_K>0$ such that
 \begin{eqnarray}\label{she-lip-space}
  \esp |\she(t,x)-\she(s,y)|^ p\leq C_K \left( |x-y|^ {p/2} + |t-s|^{p/4} \right), \quad \forall \, (t,x)\in K.
 \end{eqnarray}
 See for example (135) in \cite{Kho}. 
As $\she$ is continuous, the only possible singularities of $\she^{-1}$ correspond to zeros of $\she$. We will show that $\she(\cdot,\cdot)^{-1}$ has a continuous modification as well. We estimate
 \begin{eqnarray*}
 \esp |\she(t,x)^{-1}-\she(s,y)^{-1}|^M&=&\esp \left|\frac{\she(t,x)-\she(s,y)}{\she(t,x)\she(s,y)}\right|^M\\
 &\leq& \esp\left[ |\she(t,x)-\she(s,y)|^{2M}\right]^{1/2}\esp\left[ \she(t,x)^{-4M}\right]^{1/4}\esp\left[ \she(s,y)^{-4M}\right]^{1/4}.
\end{eqnarray*}
By (\ref{moments}), the moments of order $-4M$ are locally bounded. Together with (\ref{she-lip-space}), we conclude that, for each compact $K \subset (0,+\infty)\times \re$, there is a constant $ \tilde C_K<+\infty$, such that
\begin{eqnarray}
 \sup_{(t,x),(s,y)\in K} \esp |\she(t,x)^{-1}-\she(s,y)^{-1}|^M < \tilde C_K \left(|x-y|^{M/2}+|t-s|^{M/4}\right).
\end{eqnarray}
Hence, by Kolmogorov criterion, $\{\she(t,x)^{-1}:\, (t,x)\in K\}$ has a continuous modification $\mathcal{Y}(\cdot,\cdot)$, and hence stays bounded. It follows that $\mathcal{Y}^{-1}$ cannot assume the value $0$ in $K$. This proves (\ref{all-x-positive}).
\end{proof}
\section{Appendix: Overlap Estimates}\label{estimates-overlap}
The goal of this section is to prove the needed overlap estimates. 
Recall that $L_N(S^{(1)},S^{(2)})=\sum^{N}_{i=1} {\bf 1}_{S^{(1)}_i=S^{(2)}_i}$ and denote by $P^{(2)}$ and $E^{(2)}$ the law and expectation of two independent simple random walks.
\begin{lemma} There is a locally bounded function $\kappa(t,x)\in (0,+\infty)$ such that
 \begin{eqnarray}\label{overlap1}
  \sup_{N\geq 1} E^{(2)}\left[e^{N^{-1/2}L_{tN}(S^{(1)},S^{(2)})} |S^{(1)}_{tN}=S^{(2)}_{tN}=x\sqrt{N}\right]&\leq& \kappa(t,x),\\
  \label{overlap2}
  \sup_{N\geq 1} \tfrac{1}{\sqrt{N}}E^{(2)}\left[L_{tN}(S^{(1)},S^{(2)})e^{N^{-1/2}L_{tN}(S^{(1)},S^{(2)})}|S^{(1)}_{tN}=S^{(2)}_{tN}=x\sqrt{N} \right]&\leq& \kappa(t,x).
 \end{eqnarray}
\end{lemma}
\begin{proof}
As the estimates will be clearly uniform for $0<t\leq T$, we specify to $t=1$. 
First, note that we can reduce to consider the overlap up to time $N/2$: indeed, abbreviating $L_{m}=L_m(S^{(1)},S^{(2)})$ and recalling the notation $E^{(2)}_{x,N}[\cdot]=E^{(2)}\left[\cdot |S^{(1)}_{N}=S^{(2)}_{N}=x\sqrt{N}\right]$, simple convexity arguments yield
\begin{eqnarray*}
 E^{(2)}_{x,N}\left[e^{\beta L_N}\right]&\leq&2E^{(2)}_{x,N}\left[e^{2\beta L_{N/2}}\right] E^{(2)}_{x,N}\left[\,e^{\beta L_{N/2}}\right],\\
E^{(2)}_{x,N}\left[L_N e^{\beta L_N}\right]
 &\leq&4E^{(2)}_{x,N}\left[L_{N/2}e^{2\beta L_{N/2}}\right] E^{(2)}_{x,N}\left[\,e^{\beta L_{N/2}}\right].
\end{eqnarray*}
We will further reduce to consider the overlap of two unconditioned random walks. Let $m=N/2$. A simple application of the local limit theorem shows that there exists a constant $C>0$ such that, for all $k\geq 0$ and $x$ in a compact set,
\begin{eqnarray*}
 P^{(2)}\left[L_m=k |S^{(1)}_{tN}=S^{(2)}_{tN}=x\sqrt{N}\right]\leq Ce^{x^2}P^{(2)}\left[L_m=k\right],
\end{eqnarray*}
and, consequently,
\begin{eqnarray*}
 E^{(2)}_{x,N}\left[e^{\alpha L_m}\right]\leq Ce^{x^2}E^{(2)}\left[e^{\alpha L_m}\right],
\end{eqnarray*}
for any $\alpha \geq 0$. The problem is now reduced to estimate the local time at $0$ for the walk $Y_i=S^{(1)}_i-S^{(2)}_i$ under the law $P^{(2)}$, which is a homogeneous pinning problem.
Accordingly, we introduce some notions and results from \cite{Giamba}. Let
\begin{eqnarray}
 z_m(\beta)=E^{(2)}\left[ e^{\beta \sum^{m}_{i=1}{\bf 1}_{Y_i=0}}\right].
\end{eqnarray} 
From \cite{Giamba} (1.6)  and (2.12), it follows that there are two finite constants $c_1,\, c_2>0$ such that
\begin{eqnarray}\label{giamba-estimate}
 z_m(\beta)\leq c_1\, e^{ c_2 \beta^2 m},\quad \forall \, m\geq 1,
\end{eqnarray}
for all $\beta$ small enough. Taking $\beta=N^{-1/2}$ yields (\ref{overlap1}). For (\ref{overlap2}), all we need is a bound on the derivative of $z_m(\beta)$ with respect to $\beta$. Notice that $g(u)=z_m(u)$ is an increasing and convex function with $g(0)=1$ and
\begin{eqnarray}\label{diff-Z}
 g'(u)=E^{(2)}\left[ L_me^{u L_m}\right],
\end{eqnarray}
where $L_m=\sum^{m}_{i=1}{\bf 1}_{Y_i=0}$. By convexity, $1+ug'(u)\leq g(u)+ug'(u) \leq g(2u)$
and consequently,
\begin{eqnarray}
 \tfrac12 g'(u) \leq \frac{g(2u)-1}{2u}.
\end{eqnarray}
Together with (\ref{giamba-estimate}),
\begin{eqnarray}
 \tfrac12 \partial_u z_m(u) \leq \frac{g(2u)-1}{2u}\leq \frac{c_1 e^{4c_2mu^2}-1}{2u}\leq 4 c_3 m u e^{4c_2mu^2},
\end{eqnarray}
with $c_3=c_1c_2$. The last inequality follows from the convexity of $\exp c_2mu^2$.  Taking $u=N^{-1/2}$ and $m=N$ in the string of inequalities above ends the proof of (\ref{overlap2}). \end{proof}
\subsection*{Acknowledgements} The author would like to thank an anonimous referee and T. Albert for a very careful reading of the manuscript, F. Caravenna for valuable advice on pinning estimates, J. Quastel and D. Remenik for many conversations on the SHE, and the department of Mathematics at UW-Madison for their hospitality during his post-doctoral training.

\end{document}